\documentclass[12pt]{article}
\usepackage{amsmath,amsthm,amssymb,hyperref,setspace,geometry}
\title{A further generalization of the Glauberman-Thompson $p$-nilpotency criterion in fusion systems}
\author{Zhencai Shen, Baoyu Zhang\thanks{Corresponding author.
\newline \hspace*{0.62cm}{\it Email addresses}: zhencai688@sina.com (Z. Shen); baoyuzhang.math@outlook.com (B. Zhang).}\\
\quad
\\
{\footnotesize College of Science,
China Agricultural University,
Beijing 100083, China}}
\date{}
\newtheorem{theorem}{Theorem}[section]
\newtheorem{theoremAlph}{Theorem}

\newtheorem{lemma}[theorem]{Lemma}

\newtheorem{corollaryAlph}[theoremAlph]{Corollary}
\theoremstyle{definition}

\newtheorem{remark}[theorem]{Remark}

\expandafter\let\expandafter\oldproof\csname\string\proof\endcsname
\let\oldendproof\endproof
\renewenvironment{proof}[1][\proofname]{%
  \oldproof[\bfseries\scshape #1]%
}{\oldendproof}

\usepackage{stackengine,scalerel}
\stackMath
\def\trianglelefteqslant{\ThisStyle{\mathrel{%
  \stackinset{r}{.75pt+.15\LMpt}{t}{.1\LMpt}{\rule{.3pt}{1.1\LMex+.2ex}}{\SavedStyle\leqslant}%
}}}

\renewcommand{\unlhd}{\trianglelefteqslant}
\renewcommand{\le}{\leqslant}

\renewcommand{\leq}{\leqslant}

\begin{document}
\maketitle
\begin{abstract}
Let $p$ be a prime and $\mathcal{F}$ be a saturated fusion system over a finite $p$-group $P$. The fusion system $\mathcal{F}$ is said to be nilpotent if $\mathcal{F}=\mathcal{F}_{P}(P)$. We provide new criteria for a saturated fusion system $\mathcal{F}$ to be nilpotent, which may be viewed as extending the Glauberman-Thompson $p$-nilpotency criterion to fusion systems.

\medskip

\end{abstract}
\section{Introduction}
Throughout this paper, all groups are finite and $\mathcal{F}$ always denotes a saturated fusion system over a finite $p$-group $P$.
 We shall adhere to the notation and terminologies used in \cite{AKO, BLO, P}.

It is always worth noting that Thompson's normal $p$-complement theorems vitalized finite group
theory in the second half of the last century and opened up significant perspectives which are still powerful and active nowadays.
 Let $p$ be an odd prime. The $p$-nilpotency theorem of Glauberman and Thompson asserts that a finite group $G$ is $p$-nilpotent 
 if and only if $N_{G}(Z(J(P)))$ is $p$-nilpotent. In \cite{KL}, R. Kessar and M. Linckelmann generalize the Glauberman-Thompson $p$-nilpotency theorem and Glauberman's $ZJ$-theorem to fusion systems. Following in the 
 spirit of \cite{KL}, many definitions and results have been generalized to fusion systems, which concentrates on how the impact of certain structures on finite groups translates into influence on the global structure of fusion systems.

Let $K\leq H$ be subgroups of a finite group $G$. The subgroup $K$ is called {\it strongly closed} in $H$ with respect to $G$ if $K^G\cap H \leq K$. Strong closure is one of the essential ingredients of finite group theory. Goldschmidt's theorem on strongly closed abelian $2$-subgroups \cite{Gol} played a fundamental role in the classification of finite simple groups. This concept also facilitates the development of other aspects of algebra.
In a fusion system $\mathcal {F}$ over a $p$-group $P$, a subgroup $Q$ of $P$ is said to
 be {\it $\mathcal{F}$-strongly closed} in $P$ if no element of $Q$ is $\mathcal {F}$-conjugate to an element in $P\setminus Q$. In this paper, new nilpotency criteria for fusion systems are derived with $\mathcal{F}$-strongly closed subgroups.
 
The following theorem is a recent generalization of the Glauberman-Thompson $p$-nilpotency theorem by K\i zmaz.
 
\begin{lemma}[{{\cite{KI}}}]\label{1.1}
Let $G$ be a finite group and $P$ a Sylow $p$-subgroup of $G$, where $p$ is an odd prime. Assume that $D$ is a strongly closed subgroup in $P$. Then $G$ is a $p$-nilpotent group if and only if $N_{G}(Z(J(D)))$ is a $p$-nilpotent group.
\end{lemma}

We extend the result of K\i zmaz to fusion systems and give the following theorem.

\begin{theoremAlph}\label{A}
Let $\mathcal {F}$ be a saturated fusion system over a finite $p$-group $P$ for an odd prime $p$, and $D$ be an $\mathcal{F}$-strongly closed subgroup. Then $\mathcal {F}=\mathcal {F}_{P}(P)$
if and only if $N_{\mathcal {F}}(Z(J(D)))=\mathcal {F}_{P}(P)$.
\end{theoremAlph}

The other main theorem of this paper states as follows:

\begin{theoremAlph}\label{B}
Let $\mathcal {F}$ be a saturated fusion system over a finite $p$-group $P$ and $p$ a prime. Then $\mathcal {F}=\mathcal {F}_{P}(P)$ if and only if $N_{\mathcal {F}}(P) =\mathcal {F}_{P}(P)$ and $\Phi(P)$ is an $\mathcal{F}$-strongly closed subgroup.
\end{theoremAlph}

Since $\mathcal{F}_P(P)$ with $P\in {\rm Syl}_p(G)$ is a saturated fusion system, Corollary \ref{C} can be verified immediately.

\begin{corollaryAlph}\label{C}
Let $G$ be a finite group, $P\in{\rm Syl}_{p}(G)$ and $p$ a prime. Then $G$ is $p$-nilpotent if and only if $N_{G}(P)$ is $p$-nilpotent and $\Phi(P)$ is a strongly closed subgroup of $G$.
\end{corollaryAlph}
\section{Proof}

\begin{lemma}[{\cite[Theorem 3.7]{OS}}]\label{OS3.7}
Let $\mathcal{F}$ be a saturated fusion system over a finite $p$-group $P$ and $p$ a prime. Then $\mathcal{F}=\langle PC_\mathcal{F}(O_p(\mathcal{F})), N_\mathcal{F}(O_p(\mathcal{F})C_P(O_p(\mathcal{F})))\rangle$.
\end{lemma}

\begin{proof}[Proof of Theorem \ref{A}]
The necessity is clear. Suppose that the converse is not true, and let $\mathcal{F}$ be a minimal counterexample with respect to the number $|\mathcal{F}|$ of morphisms of $\mathcal{F}$. We will show that $\mathcal{F}$ is constrained by the following steps.
\begin{itemize}
\item[(1)] Any proper fusion subsystem $\mathcal{E}$ of $\mathcal{F}$ over $P$ is equal to $\mathcal{F}_P(P)$.

$\mathcal{F}_P(P)\subseteq N_\mathcal{E}(Z(J(D)))\subseteq N_\mathcal{F}(Z(J(D)))=\mathcal{F}_P(P)$ implies that $N_\mathcal{E}(Z(J(D)))=\mathcal{F}_P(P)$, and $D$ is strongly closed in $P$ with respect to $\mathcal{F}$ then with respect to $\mathcal{E}$. By the minimality of $\mathcal{F}$, $\mathcal{E}=\mathcal{F}_P(P)$.

\item[(2)] $O_p(\mathcal{F})\neq 1$.

Since $\mathcal{F}$ is saturated and not nilpotent, by Alperin's fusion theorem there is some fully $\mathcal{F}$-normalized subgroup $T$ of $P$ such that $N_\mathcal{F}(T)$ is not nilpotent. Therefore we can choose, among all nontrivial subgroups $U$ of $P$ with $N_\mathcal{F}(U)$ not nilpotent, one with $|N_P(U)|$ maximal.

It suffices to show that $U$ is normal in $\mathcal{F}$. Assume that $N_\mathcal{F}(U)$ is a proper subsystem of $\mathcal{F}$. Since $N_\mathcal{F}(U)$ is not nilpotent, we have $N_P(U)<P$ by (1). It follows that $$N_P(U)<N_P(N_P(U))\le N_P(J(N_P(U)))\le N_P(Z(J(N_P(U)))).$$
Further note that $Z(J(N_P(U)))>1$ since $U>1$. Now $N_\mathcal{F}(Z(J(N_P(U))))$ is nilpotent by the choice of $U$. Then $N_{N_\mathcal{F}(U)}(Z(J(N_P(U))))$ is also nilpotent as a subsystem. Since $N_P(U)$ is $N_\mathcal{F}(U)$-strongly closed and $N_\mathcal{F}(U)$ is not nilpotent by the choice of $U$, $N_\mathcal{F}(U)$ is a counterexample to the theorem. This contradicts the minimality of $\mathcal{F}$. Thus $U$ is normal in $\mathcal{F}$ and consequently $1<U\le O_p(\mathcal{F})$. Now set $Q=O_p(\mathcal{F})$.

\item[(3)] $PC_\mathcal{F}(Q)=\mathcal{F}_P(P)$.

Assume that $PC_\mathcal{F}(Q)\neq \mathcal{F}_P(P)$. Then $PC_\mathcal{F}(Q)=\mathcal{F}$ by (1).
It follows from \cite[Proposition 3.4]{KL} that $\mathcal{F}/Q$ is not nilpotent since $\mathcal{F}$ is not nilpotent. By the minimality of $\mathcal{F}$, $\mathcal{F}/Q$ is not a counterexample to the theorem. Since $P/Q$ is $\mathcal{F}/Q$-strongly closed, we have that $N_{\mathcal{F}/Q}(Z(J(P/Q)))$ is not nilpotent as $\mathcal{F}/Q$ is not nilpotent. Let $E$ be the preimage of $Z(J(P/Q))$ in $P$. Then $N_\mathcal{F}(E)$ is not nilpotent. But this forces $\mathcal{F}=N_\mathcal{F}(E)$. It is a contradiction since $E>Q=O_p(\mathcal{F})$.

\item[(4)] $Q$ is $\mathcal {F}$-centric.

Let $R=QC_{P}(Q)$. It follows that $Q\leq R \unlhd P$. If  $Q<R$, then $N_\mathcal{F}(R) = \mathcal{F}_P(P)$. By Lemma \ref{OS3.7} and (3), we have that $\mathcal{F}=\langle PC_\mathcal{F}(Q), N_\mathcal{F}(R)\rangle = \mathcal{F}_P(P)$, a contradiction. Thus  $Q=R$ and $C_{P}(Q)\leq Q$.
\end{itemize}

In view of (2) and (4), we obtain that $\mathcal{F}$ is constrained. By the model theorem (cf. \cite[I.4.9]{AKO} or \cite{BCGLO}), there exists a finite group $G$ such that $\mathcal{F}=\mathcal{F}_P(G)$, $P\in {\rm Syl}_p(G)$, $Q\unlhd G$ and $C_{G}(Q)\leq Q$. Hence $\mathcal{F}_P(P)=\mathcal{F}_P(N_G(Z(J(D))))$, where $D$ is strongly closed in $P$ with respect to $G$. By Lemma \ref{1.1}, $G$ is $p$-nilpotent, and therefore $\mathcal{F}=\mathcal{F}_P(G)=\mathcal{F}_P(P)$, which is the final contradiction.
\end{proof}

\begin{remark}
Theorem \ref{B} can be obtained by mimicking the proof of \cite[Main Theorem]{S} and we therefore omit the details. When dealing with $p=2$, we use the elementary focal subgroups $E_\mathcal{F}^p(P)$ and $E_{N_\mathcal{F}(P)}^p(P)$ instead of the focal subgroups $A_\mathcal{F}^p(P)$ and $A_{N_\mathcal{F}(P)}^p(P)$ as in \cite{S}. Then similar arguments yield $E_{N_\mathcal{F}(P)}^p(P)=\Phi(P)=E_\mathcal{F}^p(P)$ and it follows from \cite[Corollary 1.2]{DGPS} that $\mathcal{F}=\mathcal{F}_P(P)$.
\end{remark}
\section*{Acknowledgement}
The research of the work is supported in part by the Natural Science Foundation of China (No. 12071181). The authors are grateful to A. Glesser and M.Yasir K\i zmaz for many valuable suggestions that help the authors improve the manuscript.

\end{document}